\documentclass[12pt]{amsart}
\usepackage{geometry} 
\usepackage{microtype} 

\usepackage{mathtools}
\usepackage{amscd}
\usepackage{cite}
\makeatletter
\newcommand{\citenoadjust}[1]{{\let\cite@adjust\empty#1}}
\makeatother

\usepackage{hyperref}
\hypersetup{
  colorlinks   = true, 
  urlcolor     = blue, 
  linkcolor    = blue, 
  citecolor   = red 
}
\usepackage{graphicx}

\DeclareMathOperator{\sgn}{sgn}

\DeclareMathOperator{\bd}{bd} \DeclareMathOperator{\cl}{cl}

\DeclareMathOperator{\dist}{dist}

\newtheorem{theorem}{Theorem}[section]

\newtheorem{fcon}[theorem]{False Conjecture}
\newtheorem{corollary}[theorem]{Corollary}

\theoremstyle{definition}
\newtheorem{definition}[theorem]{Definition}

\theoremstyle{remark}
\newtheorem{remark}[theorem]{Remark}

\numberwithin{equation}{section}

\newcommand{\pt}{{\rm pt}}
\newcommand{\id}{{\rm id}}
\newcommand{\inte}{\mathop{\rm int}}

\renewcommand{\epsilon}{\varepsilon}
\renewcommand{\phi}{\varphi}

\newcommand{\Sg}{\Sigma}

\renewcommand{\epsilon}{\varepsilon}
\renewcommand{\phi}{\varphi}

\DeclareMathOperator{\vol}{vol}
\newcommand{\Reals}{\mathbb{R}}
\newcommand{\ZZ}{\mathbb{Z}}
\newcommand{\Sphere}{\mathbb{S}}

\title{Convex equipartitions: the spicy chicken theorem}

\author{Roman Karasev}
\author{Alfredo Hubard}
\author{Boris Aronov}

\address[R.~Karasev]{Dept.\ of Mathematics, Moscow Institute of Physics and Technology, Institutskiy per.\ 9, Dolgoprudny, Russia 141700}
\address[R.~Karasev]{Institute for Information Transmission Problems RAS, Bolshoy Karetny per. 19, Moscow, Russia 127994}
\email[R.~Karasev]{r\_n\_karasev@mail.ru}

\address[A.~Hubard]{Universit\'e Paris-Est Marne-la-Vall\'ee}
\email[A.~Hubard]{alfredo.hubard@u-pem.fr}

\address[B.~Aronov]{Department of Computer Science and Engineering, Tandon School of Engineering, New York University, Brooklyn, NY 11201 USA}
\email[B.~Aronov]{boris.aronov@nyu.edu}

\begin{document}
\begin{abstract}
We show that, for any prime power $n$ and any convex body $K$ (i.e., a compact convex set with interior) in $\Reals^d$, there exists a partition of $K$ into $n$ convex sets with equal volumes and equal surface areas. Similar results regarding equipartitions with respect to continuous functionals and absolutely continuous measures on convex bodies are also proven. These include a generalization of the ham-sandwich theorem to arbitrary number of convex pieces confirming a conjecture of Kaneko and Kano, a similar generalization of perfect partitions of a cake and its icing, and a generalization of the Gromov--Borsuk--Ulam theorem for convex sets in the model spaces of constant curvature. 
\end{abstract}
\maketitle

\section{Introduction}

Imagine that you are cooking chicken at a party. You will cut the raw chicken fillet with a sharp knife, marinate each of the pieces in a spicy sauce, and then fry the pieces. The surface of each piece will be crispy and spicy. Can you cut the chicken so that all your guests get the same amount of crispy crust and the same amount of chicken?\footnote{Vegetarian readers are welcome to substitute the chicken filet with a peeled potato.} Thinking of two-dimensional ``convex chickens,'' Nandakumar and Ramana Rao \cite{nara2008} asked the ``interesting and annoyingly resistant question'' \cite{barblsz2010} of whether \emph{a convex body in the plane can be partitioned into $n$ convex regions with equal areas and equal perimeters}.  This is easy for $n=2$ and known for $n=3$, see \cite{barblsz2010}. We confirm this conjecture and its natural generalization to higher dimensions for $n$ a prime power:

\begin{corollary}
\label{spicy-chicken}
Given a convex body $K$ in $\Reals^d$, a prime $p$, a positive integer $k$, it is possible to partition $K$ into $n=p^k$ convex bodies with equal $d$-dimensional volumes and equal $(d-1)$-dimensional surface areas.
\end{corollary} 

In fact, we derive this result from the following much more general one. Let $M^d(\kappa)$ be the $d$-dimensional simply-connected Riemannian manifold of constant curvature $\kappa$, i.e., hyperbolic space ($\kappa=-1$), Euclidean space ($\kappa=0$), or the round sphere ($\kappa=1$). Let $\mathcal{K}^d(\kappa)$ be the space of its geodesically convex sets with the Hausdorff metric.

\begin{remark}For simplicity, we state our results for absolutely continuous measures; by a standard limiting argument this implies the same results for weak limits of absolutely continuous measures substituting equalities by inequalities (e.g., for finitely supported measures). If we further assume that a limiting measure $\mu$ has the property that for every geodesic hyperplane $H$, we have $\mu(H)=0$, then our results hold with equalities (e.g., for the volume of a compact convex set).\end{remark}

\begin{theorem}
\label{extra}
Given an absolutely continuous finite measure $\mu$ on $M^d(\kappa)$, a convex body $K\in \mathcal{K}^d(\kappa)$, a family of $d-1$ continuous functionals $\phi_1,\phi_2, \ldots, \phi_{d-1} \colon \mathcal{K}^d(\kappa) \to \Reals$, a prime number $p$, and a positive integer $k$,
there is a partition of $K$ into $n = p^k$ convex bodies $K_1, K_2 \ldots K_n$, such that  
\[
  \mu(K_i)=\frac{\mu(K)}{n}
\]
and 
\[
  \phi_j(K_1)=\phi_j(K_2)= \ldots =\phi_j(K_{n}),
\]
for all $ 1\leq i\leq n$ and $1\leq j\leq d-1$.
\end{theorem} 

\begin{remark}
Note that, for $M^d(0)=\mathbb R^d$, if we let $\phi_i(K)$ be the $i$th Steiner measure of the convex set, i.e., the coefficient of $t^i$ in the polynomial 
\[
  P_K(t) = \mu(K+tB),
\]
then Corollary~\ref{spicy-chicken} reduces to Theorem~\ref{extra}; here $B$ is the unit ball.  In fact, from Theorem~\ref{extra} it follows that it is possible to make all Steiner measures of the parts equal at the same time. In particular, in $\mathbb R^3$ we can equalize the mean width along with the volume and the surface area. 
\end{remark}

\begin{remark}
 When the ambient space is the sphere $\mathbb S^d$ it is possible to take the whole sphere
  $\mathbb S^d$ as $K$ in this theorem, even though it is not convex.
\end{remark}


This paper is based on the preprints \cite{kar2010} and \cite{huar2010}. In both papers, generalized Voronoi diagrams (also known as power diagrams) were used in conjuction with a Borsuk--Ulam-type theorem about the nonexistence of a nowhere-vanishing $\Sigma_n$-equivariant map from the configuration space to a real space of appropriate dimension (see~Theorem~\ref{top} below) with certain $\Sigma_n$-action. 

The difference in the methods of \cite{kar2010} and \cite{huar2010} is that in \cite{huar2010} optimal transport was used as a first step and afterwards the Borsuk--Ulam-type statement implied the result. This method provides a different intuition and it might be, in principle, stronger than the one in \cite{kar2010}. On the other hand, we do not know of any instance where the method of \cite{huar2010} exhibits its superiority, while the method of \cite{kar2010} is simpler and more elegant, as it deals with the measure and the functionals in one step using the same Borsuk--Ulam-type statement in a configuration space, corresponding to a space of functions of dimension $d+1$.  In this paper we provide details of the proof in \cite{kar2010} in Section~\ref{section:sep-functions} and only sketch the approach of \cite{huar2010} in Section~\ref{section:transport}.\footnote{The reduction of Theorem~\ref{extra} to Theorem~\ref{top} was independently discovered in  \cite{huar2010} and \cite{kar2010}; it is the main contribution of these papers. The original version of Theorem~\ref{extra} in \cite{huar2010} is weaker than that in \cite{kar2010}.}
 
The argument in \cite{kar2010} generalizes an idea that Gromov used on his way to prove the waist of the sphere inequality, see \cite{gr2003} and \cite{mem2009}. In order to state the following result we define a \emph{centermap} as a continuous functional $c\colon\mathcal{K}^d(\kappa) \to M^d(\kappa)$. The so called Gromov--Borsuk--Ulam theorem of  \cite{gr2003} and \cite{mem2009} is the case when $p=2$ and $\mu$ is the $\mathrm{O}(d)$-invariant probability measure on $\mathbb S^d$.

\begin{theorem}
\label{gro}
Given a convex body $K \in \mathcal{K}^d(\kappa)$ (or the whole sphere if $M^d(\kappa)=\mathbb S^d$), an absolutely continuous finite measure ~$\mu$ on $K$, a prime $p$, a positive integer $k$, a continuous map $g\colon M^d(\kappa) \to \Reals^{d-1}$, and a continuous centermap $c$, then there exists a partition of $K$ into $n = p^k$ convex sets $K_1,K_2, \ldots K_{n}$, such
that
\[
\mu(K_i)=\frac{\mu(K)}{n},
\]
for all $i$, and
\[
g(c(K_1))=g(c(K_2))=\dots=g(c(K_{n})).
\]
\end{theorem}

This theorem does not follow directly from Theorem~\ref{extra}, but it suffices to put $\phi_i(K_j):=g_i(c(K_j))$ and follow the proof of Theorem~\ref{extra}. By taking $g$ to be a linear projection it is easy to see that this theorem is quantitatively best possible.


\bigskip
Yet another family of interesting corollaries related to the ham-sandwich theorem arises by defining the functionals by measures. The following result was proven in \cite{sob2010,kar2010}:

\begin{corollary}
\label{ham} 
Given $d$ absolutely continuous finite measures $\mu_1, \mu_1,\ldots \mu_d$ on $\Reals^d$, and \emph{any} number $n$ there is a partition of $\Reals^d$ into convex regions $K_1,K_2, \ldots K_n$ with $\mu_i(K_j)=\frac{1}{n}$ for all $i$ and $j$ simultaneously.
\end{corollary}

This result was conjectured by Kaneko and Kano \cite{kaka2002} who also proved the planar version. The proof was found independently in \cite{sob2010} and \cite{kar2010}. The proof in \cite{sob2010} is a variation of the argument of the first version of \cite{huar2010} and it has the nice feature of using much more basic algebraic topology, namely the classical Borsuk--Ulam theorem for $\mathbb Z/p$ actions on the sphere. 

With the full power of Theorem~\ref{extra}, the proof of this corollary is very simple and we now sketch it. 
\begin{proof}[Proof of Corollary~\ref{ham}]
Write $n=p_1^{\alpha_1} p_2^{\alpha_2}...p_k^{\alpha_k}$ and apply Theorem~\ref{extra} with $\mu_d$ as the measure, with $\phi_i(A):=\mu_i(A)$ and ${p_1}^{\alpha_1}$. For each cell of the partition $K_j$, apply the theorem again to properly renormalized measures 
\[
  \mu_i '(A):={p_1}^{\alpha_1} \int_{K_j} A\; d\mu_i
\]
to partition each cell into $p_2^{\alpha_2}$ subcells and continue in this manner.
\end{proof}

\begin{remark}
By considering $d+1$ measures, it is easy to observe that this corollary, and hence Theorem~\ref{extra}, is quantitatively best possible.
\end{remark}

The following theorem does not follow directly from Corollary~\ref{ham} because of some discontinuity issues, but is proven in a similar manner. This is a higher-dimensional generalization of the results about perfect partitions in the plane, see \cite{akknrtu2004}. Its proof is contained in Section~\ref{section:sep-functions}.

\begin{theorem}
\label{vol-area-spl} 
Suppose $K\in \mathcal{K}^d(\kappa)$ is a convex body, and, for some $1\le m \le d$, we have $m$ absolutely continuous finite measures $\mu_1,\ldots, \mu_m$ on $K$, and $d-m$ absolutely continuous finite measures $\sigma_1,\ldots, \sigma_{d-m}$ on $\partial K$. Then, for any $n$, the body $K$ can be partitioned into $n$ convex parts $K_1,\ldots, K_n$, such that, for any $i=1,\ldots, m$,
\[
  \mu_i(K_1) = \dots = \mu_i(K_n),
\]
and, for every $i = 1,\ldots, d-m$,
\[
  \sigma_i(K_1\cap\partial K) = \dots = \sigma_i(K_n\cap\partial K).
\]
\end{theorem}

In Section~\ref{section:one-dim} we show a result about upper envelopes of families of analytic functions that was inspired by the Alon--Akiyama splitting necklace theorem.

In Section~\ref{section:top} we provide a detailed proof of the Borsuk--Ulam-type statement (Theorem~\ref{top} below) that is the main topological tool used in this paper:



Let $F_n(\mathbb R^d)$ be the space of ordered $n$-tuples of pairwise distinct points in $\mathbb R^d$, i.e., $F_n(\mathbb R^d):=\{(x_1,x_2,\ldots x_n)\in \mathbb{R}^{nd}: x_i\neq x_j \textrm{ for all } i\neq j\}$. This is the classical \emph{configuration space}. The symmetric group $\Sigma_n$ acts naturally on $F_n(\mathbb R^d)$ by permuting the points in a tuple and on $\mathbb{R}^n$ by permuting the coordinate axes.  If we restrict this action on $\mathbb{R}^n$ to the orthogonal complement of the diagonal we obtain a free action of $\Sigma_n$.  
This complement of the diagonal consists of the vectors in $\mathbb R^n$ with zero coordinate sum; we denote this $(n-1)$-dimensional representation of $\Sigma_n$ by $\alpha_n$.

\begin{theorem}[D.~Fuchs, V.~Vassiliev, R.~Karasev]
\label{top}
Let $p$ be a prime, $k$ a positive integer, and $n=p^k$.  For any $\Sigma_n$-equivariant map $f\colon F_n(\Reals^d) \to \alpha_n^{\oplus (d-1)}$, there exists a configuration $\bar x\in  F_n(\Reals^d)$, such that $f(\bar x)=0$. 
\end{theorem}

Our proof is a variation of the original one found by D.~Fuchs in the case $d=2$, $p=2$, by V.~Vassiliev in the case $d=2$ for any prime~$p$, and by the first-named author of this paper for the prime power case. The proof presented here avoids using the Euler class and Poincar\'e duality. Instead, it relies completely on homological considerations. While we provide background on homology with twisted coefficients and give a very detailed description of the Fuchs cell decomposition of configuration space, the actual proof of this theorem is rather short. One only needs to exhibit a generic section whose zero set is not null homologous in the compact support homology with twisted coefficients and then show that any two such zero sets are homologous. 

In \cite{bz2012} P.~Blagojevi\'c and G.~Ziegler give another proof of this statement using obstruction theory. They retract the configuration space onto a compact polyhedron contained in it. Their technique has the advantage of providing a converse result, namely, if $n$ is not a prime power, then there exists a nonvanishing equivariant map, $f\colon F_n(\Reals^d) \to \alpha_n^{\oplus (d-1)}$. This observation shows that the open cases of the Nandakumar--Ramana Rao conjecture ($d=2$, $n=6$, for example) are resistant to known topological methods.

\section{Equipartitions via Optimal Transport}
\label{section:transport}

In this section we sketch the relation of power diagrams to optimal transport, and to the problems at hand. 

Let $\mu$ and $\nu$ be measures on $\Reals^d$.  For a transformation
$T\colon \Reals^d \to \Reals^d$, let $T_\#(\mu)$ denote the push forward of
$\mu$. Consider the following Monge-Kantorovich \emph{optimal
  transport problem} with quadratic cost
\[
  \inf_{\substack{T \colon \Reals^d \to \Reals^d\\T_\#(\mu)=\nu}} \int |x-T(x)|^2\; d\mu(x).
\]
If we choose $\mu$ to be an absolutely continuous probability measure and $\nu$ to be a probability measure supported on a finite set, then this infimum is achieved by a measurable transformation~$T_\mu^\nu$.  Moreover, the map can be described by a generalized Voronoi diagram based at the support of $\nu$. To be precise, for an $n$-tuple of pairwise distinct points (\emph{sites}) $x_1,x_2, \ldots x_n \in \Reals^d$ with corresponding \emph{radii} $r_1,r_2, \ldots r_n \in \Reals$, the \emph{power diagram} is a tessellation of $\Reals^d$ which generalizes the Dirichlet--Voronoi diagram. A point $x\in \Reals^d$ is assigned to the cell $C_i$ corresponding to the site $x_i$ if $f_i(x)=|x-x_i|^2-r_i$ is minimal among all $i$'s.\footnote{The term ``power'' comes from Euclidean geometry. Recall that the power of a point $p$ with respect to a circle of radius $r$ and center $y$, that does not contain $p$, is $|p-y|^2-r^2$.} Notice that cells can be empty and when they are not, sites need not be contained in their corresponding cells. The following theorem from \cite{aha1998,bren1991,mac1995} relates optimal transport to power diagrams:

\begin{theorem}[L.~Kantorovich, Y.~Brenier, R.~McCann]
\label{aha1998}
Let $\mu$ be an absolutely continuous probability measure and $\nu$ be a convex combination of delta masses at $n$~distinct points $\langle x_1,x_2, \ldots, x_n\rangle$ in $\Reals^d$ 
There exists a set of radii $\langle r_1, r_2, \ldots, r_n \rangle$ such that the power diagram $\langle C_1, C_2, \ldots, C_n\rangle$ defined by $\langle x_1,x_2, \ldots, x_n\rangle$ and $\langle r_1, r_2, \ldots, r_n \rangle$ gives the unique solution to the optimal transport problem, i.e.,
\[
(T_\mu^\nu)^{-1} (x_i)=C_i ,
\]
up to a set of measure zero.
\end{theorem}

It is not hard to see that the parameters $T_\mu^\nu$, $\{r_i\}$, and the cells $C_i$ are essentially uniquely determined and depend continuously on the measure $\nu$, see \cite{vil2009} for a thorough exposition.

\subsection{The idea of the optimal transport approach to equipartitions}

Consider an ordered $n$-tuple $\bar{x}=\langle x_1,x_2, \ldots x_n\rangle$ of pairwise distinct points and a convex body~$K$.  The solution of the optimal transport problem in which the source measure is the volume of~$K$ and the target measure is $\nu_{\bar{x}}=\frac{1}{n}\sum \delta_{x_i}$ induces a convex partition of $K$ into convex sets of equal volume.  As the points $\langle x_1,x_2, \ldots x_n\rangle$ move, the surface areas of the corresponding cells change continuously. Since the measure  $\nu_{\bar{x}}$ is $\Sigma_n$-invariant, the surface areas define a $\Sigma_n$-equivariant map to $\Reals^{n}$ (with the permutation action of $\Sigma_n$) and by Theorem~\ref{top} this map intersects the diagonal, which is the kernel of the natural projection $\Reals^n\to \alpha_n$. 

\section{Equipartitions via measure-separating functions}
\label{section:sep-functions}

Consider a compact topological space $X$ with a Borel probability measure $\mu$. Let $C(X)$ denote the set of real-valued continuous functions on $X$. 


\begin{definition}
A finite-dimensional linear subspace $L\subset C(X)$ is called \emph{measure separating}, if, for any $f\neq g\in L$, the measure of the set 
\[
  e(f, g) = \{x\in X : f(x) = g(x)\}
\]
is zero.
\end{definition}

In particular, if $X$ is a compact subset of $\mathbb R^d$, or 
of other real analytic smooth manifold, such that $X=\cl(\inte X)$ and $\mu$ is any absolutely continuous measure, then any finite-dimensional space of real-analytic functions is measure separating, because the sets $e(f,g)$ always have dimension less than $d$ and therefore measure zero. 

For a finite subset of a measure-separating subspace we define a partition of~$X$, as follows.  Suppose $F=\{u_1, \ldots, u_n\} \subset C(X)$ is a family of functions such that $\mu (e(u_i,u_j)) = 0$ for all $i\neq j$. The sets (some of which may be empty)
\[
  V_i = \{x\in X : \forall j\neq i\ u_i(x)\ge u_j(x)\}
\]
have a zero-measure overlap, so they define a partition $P(F)$ of $X$. Notice that in the case where $u_i$ are linear functions on $\mathbb R^d$, $P(F)$ is a power diagram.
As a warm up, we prove a result on splitting measures that is enough to derive Corollary~\ref{ham}:

\begin{theorem}
\label{gen-spl}
Suppose $L$ is a measure-separating subspace of $C(X)$ of dimension $d+1$, $\mu_1,\ldots, \mu_d$ are absolutely continuous (with respect to the original measure on $X$) finite measures on $X$. Then, for any prime power $n$, there exists an $n$-element subset $F\subset L$ such that, for every $i=1,\ldots,d$, 
the family $P(F)$ partitions the measure $\mu_i$ into $n$ equal parts.
\end{theorem}

\begin{proof}
Let $F \in F_n(L)$; as above the configuration space has the natural action of the symmetric group $\Sg_n$ by relabeling the functions.

For every $i=1,\ldots, d$ and $P(F) = \{V_1,\ldots, V_n\}$ the values 
\[
  \mu_i(V_1) - \frac{1}{n}, \ldots, \mu_i(V_n) - \frac{1}{n}
\]
define a map $f_i \colon F_n(L)\to \alpha_n\subset\Reals^n$ (because the sums of coordinates are adjusted to be zero).  This map is $\Sg_n$-equivariant, and from the absolute continuity and the measure-separation property we deduce that the map $f_i$ is continuous.

Consider the $\Sigma_n$-equivariant direct sum map 
\[
  f=f_1\oplus\dots\oplus f_d \colon F_n(L) \to \alpha_n^{\oplus d}.
\]
By Theorem~\ref{top}, it has to vanish at some configuration, which is precisely what we needed to prove.
\end{proof}

Theorems~\ref{extra} and~\ref{vol-area-spl} do not follow directly from the above theorem. The problem lies in the fact that cells in power diagrams might vanish giving rise to 
discontinuities of the constructed maps. However, we only need a minor modification to the previous proof.

\begin{proof}[Proof of Theorem~\ref{vol-area-spl}]
Let $\mathbb S^d$, $\Reals^d$, or $\mathbb H^d$ be embedded in $\Reals^{d+1}$ in the usual way. That is, the round sphere $\mathbb S^d=\{x \in \Reals^{d+1}: \langle x,x\rangle=1\}$, the affine hyperplane $\Reals^d=\{x\in \Reals^{d+1}: x_{d+1}=1\}$, or the pseudosphere $\mathbb H^d=\{x \in \Reals^{d+1}: \langle x,x\rangle_L=-1\}$, where $\langle \cdot,\cdot\rangle_L$ denotes the Lorenzian quadratic form $\sum_{i=1}^d x_i^2-x_{d+1}^2$. 

Let $L$ be the $(d+1)$-dimensional space of real-valued homogeneous linear functions on $\mathbb R^{d+1}$ restricted to $M=\mathbb S^d,\Reals^d, \mathbb H^d$ respectively.  Notice that the crucial fact is that the intersection of $M$ with any zero set of $f\in L$ is a geodesic hyperplane in $M$. 

For any $F\in F_n(L)$ consider its partition $P(F) = \{V_1,\ldots, V_n\}$, restricted to $M$. This is actually a partition into convex parts, since the walls of the partitions are geodesic hyperplanes.

Define the maps from $F_n(L)$ to $\alpha_n$ by 
\[
  f_i \colon F\mapsto \left(\mu_i(V_1\cap K) - \frac{1}{n}\mu_i(K), \ldots, \mu_i(V_n\cap K) - \frac{1}{n}\mu_i(K)\right),
\]
for $i=1,\ldots, m$.
They are continuous on the whole $F_n(L)$. Let $Z\subset F_n(L)$ consist of configurations $F$ such that $f_i(F)=0$ for all $i=1,\ldots, m$. For $F\in Z$ the sets $V_1(F), \ldots, V_n(F)$ are nonempty and have nonempty interior, where we put $V_j(F) := V_j(F)\cap K$, for brevity.


Now the remaining maps
\[
  f_i \colon F\mapsto \left(\sigma_{i-m}(V_1\cap\partial K) - \frac{1}{n}\sigma_{i-m}(\partial K), \ldots, \sigma_{i-m}(V_n\cap\partial K) - \frac{1}{n}\sigma_{i-m}(\partial K)\right),
\]
are defined on $Z$, for $i=m+1, \ldots, d$. Note that for $F\in Z$ (and in some neighborhood of $U\supset Z$) any two convex sets $V_j(F), V_l(F)$ are separated by a hyperplane $u_j(x) = u_l(x)$, and since $V_j(F)$ and $V_l(F)$ have nonempty interiors this hyperplane is transverse to $\partial K$. Therefore the sets $V_j(F)\cap\partial K$ depend continuously on $F\in U$ and the rest of the proof for a prime power $n$ is similar to the previous proof.

Thus the case when $n$ is a prime power is done.  Otherwise, we proceed as in the proof of Corollary~\ref{ham}, that is we write  $n=p_1^{\alpha_1} p_2^{\alpha_2}...p_k^{\alpha_k}$ and at the $(i+1)$th step we apply the theorem with $p_{i+1}^{\alpha_{i+1}}$ to each of the convex pieces obtained at the $i$th step and appropriately renormalized measures. 
\end{proof}
 
\begin{proof}[Proof of Theorems~\ref{extra} and~\ref{gro}]
Generally, we proceed as in the previous proof. For $i=1, \ldots, d-1$, we define the maps 
\[
  f_i \colon F_n(L) \to \alpha_n
\]
as follows. For $F\in F_n(L)$ and $P(F) = \{V_1,\ldots, V_n\}$ put 
\[
  m_i(F) = \frac{1}{n} \sum_{j=1}^n \phi_i(V_j(F)\cap K),
\]
and
\[
  f_i \colon F \mapsto \big(\phi_i(V_1(F)\cap K), \ldots, \phi_i(V_n(F)\cap K)\big) - \big(m_i(F), \ldots, m_i(F)\big).
\]
Define the map $f_d$ as before
\[
  f_d \colon F\mapsto \left(\mu(V_1\cap K) - \frac{1}{n}\mu(K), \ldots, \mu(V_n\cap K) - \frac{1}{n}\mu(K)\right).
\]

Note that the maps $f_1,\ldots, f_{d-1}$ are defined only for $F$ such that all the sets $\{V_j(F) \cap K \}_{j=1}^n$ are nonempty. Moreover, these maps may be discontinuous. To correct this, consider the closed subset $Z\subseteq F_n(L)$ consisting of configurations $F$ such that $f_d(F)=0$. For $F\in Z$ the sets $V_j(F) \cap K$ have equal measures, and therefore they are convex compact sets with nonempty interiors (convex bodies), and they depend continuously (in the Hausdorff metric) on $F$, because their facets depend continuously on $F$. Now assume that the maps $f_1, \ldots, f_{d-1} \colon Z\to \alpha_n$ are defined by the above formulas, and extend each map $f_i$ ($1\le i \le d-1$)  separately to a continuous $\Sg_n$-equivariant map $f_i \colon F_n(L)\to \alpha_n$.  This can be done because we extend them from a closed subspace to the whole manifold. As before, Theorem~\ref{top} applied to $f$ yields the result. For the proof of Theorem~\ref{gro}, put $\phi_i:= g_i(c)$ and the same argument applies.
\end{proof}

\section{Measures on the segment and complexity of upper envelopes}
\label{section:one-dim}

Recall the ``splitting necklace'' theorem from \cite{alon1987} in its continuous version:

\begin{theorem}
\label{alon-spl}
Suppose we are given absolutely continuous finite measures $\mu_1,\ldots, \mu_d$ on a segment $[0,1]$. For an integer $n\ge 2$ put $N=d(n-1)+1$. Then $[0,1]$ can be partitioned into $N$ segments $I_1, \ldots, I_N$, and the family $\mathcal F = \{I_i\}_{i=1,\ldots, N}$ can be partitioned into $n$ subfamilies $\mathcal F_1,\ldots, \mathcal F_n$, so that, for any $i=1,\ldots, d$ and $j=1,\ldots, n$,
\[
  \mu_i\left(\bigcup \mathcal F_j\right) = \frac{1}{n}\mu_i([0, 1]).
\]
\end{theorem}

Let us try to reduce Theorem~\ref{alon-spl} to Theorem~\ref{gen-spl}.

Take $L$ to be the set of polynomials of degree at most $d$ on the segment $[0, 1]$. In this case we obtain $n$ polynomials, the sets of the partition $P(F)$ are unions of several segments, and we have to show that the total number of segments does not exceed $d(n-1)+1$. This would follow from the following claim:

\begin{fcon}
\label{pol-sup}
Suppose $f_1,\ldots, f_n$ are polynomials of degree at most $d$, for $x\in\mathbb R$ denote
\[
  g(x) = \max\{f_1(x), \ldots, f_n(x)\}.
\]
Then $g(x)$ has at most $d(n-1)$ points of switching between a pair of $f_i$'s.
\end{fcon}

\begin{remark}
The function $g(x)$ is usually called the \emph{upper envelope} of the set of polynomials.
\end{remark}

The case of non-prime-power $n$ in the splitting necklace theorem would follow from this conjecture by iterating the splittings, as in the original proof of Theorem~\ref{alon-spl}.

This conjecture is obviously true as stated for $d=1$ or $n=2$.  The latter case gives Theorem~\ref{alon-spl} in the case $n=2^k$ by iterating, but this is the same as using the ham-sandwich theorem and therefore not too interesting. The case $d=2$ can also be done ``by hand,'' ordering the polynomials by the coefficient of $x^2$ and applying induction. 

But generally Conjecture~\ref{pol-sup} is false. Arseniy~Akopyan has constructed a counterexample for $d=3$, $n\ge 4$ (private communication). An unpublished construction of P.~Shor cited in \cite{agsh1995} shows that for $d=4$ the number of ``switch'' points may grow as $\Omega(n\alpha(n))$ in $n$, where $\alpha(n)$ is the inverse Ackermann function. In \cite{agsh1995} this problem was studied in a combinatorial setting. The sequence of ``switches'' between $n$ polynomials may be encoded as a word on $n$~letters with some restrictions depending on the degree $d$; such sequences are called \emph{Davenport--Schinzel sequences}. It is known \cite{agsh1995} that the maximum length of such a word complies with Conjecture~\ref{pol-sup} for $d=1,2$; but it is asymptotically superlinear in $n$ for any fixed $d\ge 3$.

However, the following fact is known: Theorem~\ref{alon-spl} is tight and the number~$d(n-1)+1$ cannot be reduced. As a consequence, we obtain the following Erd\H{o}s--Szekeres-type theorem about real-analytic functions:

\begin{theorem}
Suppose $L\subset C^\omega[0, 1]$ is a $(d+1)$-dimensional space of functions and $n$ is a prime power. Then there exist distinct  $f_1,\ldots, f_n\in L$ such that the upper envelope
\[
  g(x) = \max\{f_1(x), \ldots, f_n(x)\}
\]
has at least $d(n-1)$ non-analytic points (``switch points'').
\end{theorem}

\begin{proof}
If, for every subset $\{f_1,\ldots, f_n\}\subset L$, the number of changes of maximum in $g(x)$ from $f_i(x)$ to $f_j(x)$ (they are exactly non-analytic points) is less than $d(n-1)$, we would prove Theorem~\ref{alon-spl} using Theorem~\ref{gen-spl} with fewer than $d(n-1)+1$ segments. But this is known to be impossible.
\end{proof}

\section{Borsuk--Ulam-type theorem for configuration spaces}
\label{section:top}

Theorem~\ref{top} was contained in \cite{kar2009}, where its proof was sketched, based on previously known facts.  In fact, the most important cases of this theorem were previously known. For $n=p$ (i.e., a prime number) this theorem is valid even in $\mathbb Z/p$-equivariant cohomology (if we embed $\mathbb Z/p \subset \Sg_p$ in the natural way). This is a particular case of \cite[Lemma~5]{kar2009}, and seems to be have been known much earlier, see \cite[Theorem~3.4, Corollaries~3.5 and 3.6]{ct1991}, for example.
 
The case $d=2$ of Theorem~\ref{top} is contained in the paper \cite{vass1988} of V.~Vassiliev; the main idea of the proof goes back to D.~Fuchs \cite{fuks1970}, who solved the simplest case of $n=2^k$ and $d=2$. This case $n=2^k$ also follows from the direct computations in \cite{hung1990}, reproduced implicitly in \cite{mem2009}. 
While Vassiliev's proof uses Euler class considerations, it seems hard to find a satisfying reference for Poincar\'e duality with twisted coefficients in the noncompact case. Hence we provide a version of the proof from \cite{kar2009} that bypasses the Euler class and is based just on homology with the usual transversality argument. We provide a proof with compact support homology, then we include a variation that assumes only standard singular homology.  We still need the twisted coefficients (see \cite{bm1960} for compact support homology and \cite{ste43} for twisted coefficients).

In the recent paper \cite{bz2012}, Theorem~\ref{top} was established without using compact support homology nor twisted coefficients by constructing an elegant finite $\Sigma_n$-equivariant model for the configuration space $F_n(\Reals^d)$.  

\subsection{The cell decomposition of the configuration space}

We will describe a cell decomposition of the one-point compactification of $F_n(\Reals^d)$, denoted by $F_n'(\Reals^d):=F_n(\Reals^d) \cup \{\pt\}$. This decomposition is labeled by elements of $\Sigma_n$ and the labeling is $\Sigma_n$-equivariant so it induces a cell decomposition of $(F_n'(\Reals^d)/ \Sigma_n,\pt)$ when we ``forget'' the labels. We will use these decompositions to perform homological calculations. These decompositions appear in Fuchs's paper \cite{fuks1970} for $d=2$ and in \cite{vass1994} for the general case.
 
The reader is encouraged to examine figure~\ref{fig:tree} and skip the next few paragraphs.
\begin{figure}
  \centering

  \caption{Trees and cells: three trees and matching example point configurations from the corresponding cells}
  \label{fig:tree}
\end{figure}

The cell decomposition of $F_n'(\Reals^d)$ has one $0$-cell that corresponds to the point at infinity. The remaining cells are in bijection with a family of ordered, labeled trees of the following form: the height of the tree is $d$, in other words, the tree has $d+1$ levels including the root. Children of every node form a linearly ordered set. The tree has $n$ leaves, all of which occur at the bottom level. Only the leaves are labeled and they are labeled with numbers $1$ through~$n$ (we can think of these labels as an element of $\Sigma_n$).

The bijection between trees and cells of the decomposition is such that the dimension of the cell corresponding to the tree $T$ is $|T|-1$, where $|T|$ denotes the number of vertices in the tree $T$. Denote by $\deg(v)$ the number of children of the vertex $v$; for example, a leaf has $\deg(v)=0$. 
 The attaching maps will be defined implicitly. Instead we describe maps from the cells to configuration space. The following sets will be convenient to describe the maps: 
\[
  \Box^k:=\{(t_1, t_2, \ldots t_k) \in [-\infty,\infty]^k : t_1\leq t_2\ldots \leq t_k\}.
\] 
Note that $\Box^k$ is homeomorphic to a closed $k$-dimensional ball. 

For each tree $T$ we define a continuous map $C_T \colon \Pi_{v \in V} \Box^{\deg(v)} \to F_n'(\Reals^d)$, where $\Pi_{v \in V} \Box^{\deg(v)}$ denotes the Cartesian product of $\Box^k$. To describe this map we recall a sorting algorithm that assigns a tree to each configuration. 


The root (level~0) is associated with the entire configuration which is an $n$-tuple of points in $\Reals^d$. We sort the points by their first coordinate; nodes on level~$1$
correspond to groups of points sharing the same coordinate.  For example, if all points in the configuration have the same first coordinate, the root has one child; if all first coordinates are different, the root has $n$ children, sorted in the order of coordinate values.  The construction proceeds recursively: on level~$2$, we consider the set of points associated with a node at level~$1$, split them into groups according to the value of their second coordinate, sorted in increasing order, and associate a level~$2$ node with each group.  Repeat the process for each level, stopping at level~$d$, where we necessarily get a total of $n$ leaves, as all points differ in at least one coordinate. Finally, label the leaves of the tree by the labels of the points of the configuration. This labeling of the leaves corresponds to the lexicographical order of the points of the configuration.

In a word, given two points $x_{i_1}$ and $x_{i_2}$ in a configuration, the closest common ancestor of the leaves labeled $i_1$ and $i_2$ represents the largest indexed coordinate in which the two points coincide. With this sorting algorithm in mind, we return to describing the map $C_T$, which is, in some sense, the inverse of the algorithm. 

There is a natural correspondence between the coordinates of $\Pi_{v \in V(T)} \Box^{\deg(v)}$ and the vertices of $T$ minus the root. For each element $q \in \Pi_{v \in V(T)} \Box^{\deg(v)}$ we think of an assignment of a real number to each non-root vertex of the tree. For example, the first $\deg(root)$ coordinates of~$q$ are assigned to the children of the root respecting the order. Given this assignment, we describe the configuration $C_T(q)$. The $j$-th coordinate of the $i$-th point of the configuration is the coordinate of $q$ assigned to the unique vertex of level $j$ on the path from the root to the leaf ~$i$. This process describes an element of $\Reals^{nd}$ corresponding to an element of $\Pi_{v \in V(T)} \Box^{\deg(v)}$, for $F_n'(\Reals^d)$ to be the image, modify this assignment to map all the elements that had been mapped to $\Reals^{nd}\setminus F_n(\Reals^d)$ to the point at infinity. It is easy to check that we have defined a cell decomposition. Note that the boundary of the $(n+d-1)$-dimensional cell, the cell of lowest possible positive dimension, is the point at infinity. The elements at the boundary of any cell are those for which the inequality between two coordinates $t_{j_1}\leq t_{j_2}$ in one of the sets $\Box^{\deg(v)}$ becomes equality or those for which some coordinate $t_j$ is $\pm \infty$. The construction guarantees that boundary points of cells of dimension larger than $n+d-1$ are mapped to the lower dimensional skeleton. This concludes the description of the cell decomposition.

This decomposition is $\Sigma_n$-equivariant.  Moreover, the only fixed point is the point at infinity. Hence it induces a fixed-point-free $\Sigma_n$-equivariant cell decomposition of the pair $(F'_n(\Reals^d)/\Sigma_n,\pt)$, where $\pt$ is the point at infinity.


We denote by $\mathbf{C}_i(F'_n(\Reals^d)/\Sigma_n,\pt)$ the chain complex corresponding to this cell decomposition.  In Section~\ref{subsection:twisted} and subsequent sections the choice of coefficients will play an important role and we will denote  $\mathbf{C}_i(F_n'(\Reals^d)/\Sigma_n,\pt; R)$  the chain complex with coefficients in $R$, a $\Sigma_n$-module. Our statements about $\mathbf{C}_i(F_n'(\Reals^d)/\Sigma_n,\pt)$ hold for any coefficient system.

\subsection{A vector bundle reformulation}

Let $\rho_n \colon \Sigma_n \to \mathrm{O}(n)$ be the standard representation of $\Sigma_n$ by permutation matrices. The diagonal $\Delta:=(t,t, \ldots t)$ is invariant under the induced action of $\Sigma_n$ on $\Reals^n$ and this representation splits into two irreducible representations, the diagonal~$\Delta$, and its orthogonal complement $\{(y_1, y_2, \ldots y_n): y_i \in \Reals, \sum y_i=0\}$.  We consider the irreducible representation on the orthogonal complement and denote it by $\alpha_n \colon \Sigma_n \to \mathrm{O}(n-1)$.

Associated to this representation there is an $(n-1)$-dimensional vector bundle $\eta$ given by $(\alpha_n \times F_n(\Reals^d))/ \Sigma_n  \to F_n(\Reals^d)/ \Sigma_n$. This bundle can also be considered as a $\Sigma_n$-equivariant bundle over $F_n(\Reals^d)$. Any $\Sigma_n$-equivariant map  $f \colon F_n(\Reals^d)\to \alpha_n$ corresponds to a section of~$\eta$. Similarly, a $\Sigma_n$-equivariant map $f \colon F_n(\Reals^d)\to \alpha_n^{\oplus (d-1)}$ corresponds to a section of the $(d-1)$-fold Whitney sum of $\eta$ with itself. Theorem~\ref{top} is equivalent to the nonexistence of nowhere-zero sections of the vector bundle $\eta^{d-1}$. 

The strategy is to show that the zero sets of two generic sections are homologous and compute their homology class by exhibiting a particular generic section.
 
\subsection{The zero set of a generic section}

Forgetting the labels, there is only one tree with $n+d$ vertices so $\mathbf{C}_{n+d-1}(F_n'(\Reals^d)/\Sigma_n,\pt)$ has only one cell generator which we will denote by~$e$. Now we exhibit a section $s_g$ of $\eta$ that is transversal to the zero section and such that the pullback of the zero section by $s_g$ is the cell
$e$. 

Consider the map $g \colon F_n(\Reals^d)\to \Reals^{n(d-1)}$ that forgets the last coordinate of every point of the configuration. This map is clearly equivariant and so, by taking the Cartesian product and the quotient, it induces a section 
\[
  s_g \colon F_n(\Reals^d)/\Sigma_n \to (F_n(\Reals^d) \times
\alpha_n^{\oplus (d-1)})/ \Sigma_n.
\]
The open manifold $s_g^{-1}(0)$ corresponds to the configurations for which all the points share the first $d-1$ coordinates---this is precisely the generating cell in $\mathbf{C}_{n+d-1}(F_n'(\Reals^d)/ \Sigma_n,\pt)$. Observe that the section is transversal to the zero section, which can be clearly seen in the coordinates. Since the covering $F_n(\Reals^d) \to F_n(\Reals^d)/ \Sigma_n$ is regular, we can equivalently show that the image of $(id,g) \colon F_n(\Reals^d) \to F_n(\Reals^d) \times \Reals^{n(d-1)}$ is transversal to the image of $(id, 0)$.  Moreover, we can identify $F_n(\Reals^d)$ with its inclusion in $\Reals^{nd}$. Now we have two linear maps and transversality follows from counting dimensions and checking that the differentials are nondegenerate. 

We can conclude that $s_g$ is transversal to the zero section. There are no trees with fewer vertices, so the kernel of the boundary map $\partial_{n+d-1}$ is the whole $\mathbf{C}_{n+d-1}(F_n'(\Reals^d)/\Sigma_n, \pt)$; geometrically, the boundary of the cell $e$ is attached to the point at infinity. To compute the $(n+d-1)$-homology group we just need to understand what is the image of the boundary operator. The cells generating $\mathbf{C}_{n+d}(F_n'(\Reals^d)/\Sigma_n, \pt)$ correspond to trees with one vertex of each level from $0$ to $d-2$, two vertices of level $d-1$, and $n$ 
vertices of level $n$. Corresponding configurations are contained in a  $2$-plane parallel to the $x_{d-1}$- and $x_d$-axes, with the points of the configuration divided into two groups, each group lying on a line parallel to the $x_d$~axis, so $\mathbf{C}_{n+d}(F_n'(\Reals^d)/\Sigma_n,\pt)$ has $n-1$ generators, one for each nontrivial solution of $n_1+n_2=n$ in positive integers. The points lie on lines $\ell_1$ and $\ell_2$ parallel to the last coordinate axis;  we assume that $\ell_1$ is lexicographically before $\ell_2$, and we let $n_1$ be the number of points on $\ell_1$ and $n_2$ be the number of points on $\ell_2$. In this way, specifying the value of $n_1$ determines a generator of $\mathbf{C}_{n+d}(F_n'(\Reals^d)/\Sigma_n,\pt)$ which we denote by $e_{n_1}$.  This set of  $n-1$ generators forms is basis of $\mathbf{C}_{n+d}(F_n'(\Reals^d)/\Sigma_n,\pt)$.   Since $\mathop{\mathrm{Ker}} \partial_{n+d-1}=\mathbf{C}_{n+d-1}(F_n'(\Reals^d)/\Sigma_n,\pt)$, understanding the $(n+d-1)$-homology boils down to understanding $\mathop{\mathrm{Im}} \partial_{n+d}$. 
We did not mention the homology coefficients up to this point, but here the choice of coefficients becomes crucial.


\subsection{The case $p=2$}

For $p=2$, we take coefficients in the field $\mathbb F_2$. The boundary map annihilates the following multiples of $e$ for $1\le n_1\le n-1$ 
\[
\partial(e_{n_1})={n \choose n_1} e.
\]
This is easy to see as follows: Consider a configuration $\bar{x}$ in the cell $e$.  This is a configuration of unlabeled points on a line $\ell$.  Every configuration is a regular value of the attaching map. Fix a cell $e_{n_1}\in\mathbf{C}_{n+d}(F_n'(\Reals^d)/\Sigma_n, \pt; \mathbb F_2)$ and note that there are exactly ${n \choose n_1}$ configurations at the boundary of $e_{n_1}$ that map to $\bar{x}$, one for each splitting of $\bar{x}$ into two nonempty sets of $n_1$ on the left line and ${n-{n_1}}$ on the right line.

Now, if $n=2^k$, then  $2\mid{n \choose n_1}$, the boundary map $\partial$ is the zero map in $\mathbb F_2$ coefficients, and so the class $[e] \in H_{n+d-1}(F_n'(\Reals^d)/\Sigma_n,\pt; \mathbb F_2)=\mathbb F_2$, represents the nontrivial element. See also \cite{fuks1970} for a related discussion.

There is an alternative approach to the case $p=2$ that can be found in \cite{gr2003} and \cite{mem2009}. Instead of looking at the full group of symmetries, restrict the action to the automorphism group of the complete binary tree of height~$k$ with~$2^k$ leaves. 
This group sits naturally inside the symmetric group. It is not hard to prove the inductive formula $Aut(T_k)=\mathbb Z/2 \wr Aut(T_{k-1})$, 
where $\wr$ denotes the wreath product. The group $Aut(T_k)$ is the Sylow subgroup $\Sigma_{2^k}^{(2^k)} \subset \Sigma_{2^k}$.  After restricting the configuration space to the \emph{wreath product of spheres} $S^{d-1}\wr\dots\wr S^{d-1}$ (see \cite{hung1990} for the details), the Euler class of $\eta^{d-1}$ with mod $2$ coefficients coincides with the top Stiefel-Whitney class, which is amenable to induction on $k$, using the knowledge about the cohomology of the wreath product of spaces. Note that the base of the induction corresponds to the Borsuk--Ulam theorem (see  \cite[page~10]{mem2009}). Generally, the map~$s_g$ has a unique transversal to zero in $S^{d-1}\wr\dots\wr S^{d-1}/\Sigma_{2^k}^{(2^k)}$, which establishes Theorem~\ref{top} immediately. 

\subsection{Homology with twisted coefficients}
\label{subsection:twisted}


We now recall what homology with twisted coefficients is. The cellular decomposition of $F'_n(\Reals^d)$ is invariant under the action of $\Sigma_n$. Moreover, all the cells except the point at infinity are permuted by $\Sigma_n$ freely. Recall that we denote the corresponding relative chain groups by $\mathbf{C}_i(F_n'(\Reals^d),\pt)$. When we want to introduce the twisted coefficients, we start with a $\Sigma_n$-module $R$. Then the corresponding equivariant twisted chains are those chains from $\mathbf{C}_i(F_n'(\Reals^d),\pt)\otimes R$ that are invariant with respect to the diagonal action of $\Sigma_n$ on this abelian group, that is
\[
  \mathbf{C}_i(F_n'(\Reals^d)/\Sigma_n,\pt; R) = \left( \mathbf{C}_i(F_n'(\Reals^d),\pt)\otimes R \right)^{\Sigma_n}.
\]
The differential in this complex is given by $\partial\otimes \id_R$.

We are going to be interested in the particular case $R=\widehat\ZZ$, where $\widehat\ZZ$ is the free abelian group $\ZZ$, but with $\Sigma_n$ acting on it as $x\mapsto \sgn\sigma\cdot x$, where $\sgn x$ is the sign of a permutation~$\sigma$. In this case the group of chains $\mathbf{C}_i(F_n'(\Reals^d)/\Sigma_n,\pt; \widehat\ZZ)$ may be regarded as the subgroup of those chains in $\mathbf{C}_i(F_n'(\Reals^d),\pt)$ that get multiplied by $\sgn \sigma$ under the action of $\sigma\in\Sigma_n$.

\subsection{Homology in dimension $n+d-1$ with twisted coefficients}


We have to check that the cycle $[e]\in \mathbf{C}_{n+d-1}(F_n'(\Reals^d)/\Sigma_n,\pt; \widehat\ZZ)$, corresponding to the unique orbit of $(n+d-1)$-dimensional cells is not annihilated by the boundary map with twisted coefficients. Here $[e]$ corresponds to the unique orbit of a $(n+d-1)$-dimensional cell $e'$; more precisely,
\[
  [e] = \sum_{\sigma\in\Sigma_n} \sgn \sigma\cdot \sigma e', 
\]
with orientation of $\sigma e'$ chosen so that the map $\sigma \colon e' \to \sigma e'$ preserves the orientation. 

The cells of dimension $n+d$ can be described as orbits of the following cells: let $e'_{n_1}$ be the cell of configurations with first $n_1$ points on $\ell_1$ and the last $n-n_1$ points on $\ell_2$, two vertical lines on a $2$-plane perpendicular to the first $d-2$ vectors of a standard basis. We also assume that the point order on $\ell_1$ and $\ell_2$ is consistent with the indexing. 

The cell $e'_{n_1}$ is explicitly given by relations (here the subscript is the index of a point in the list, and the superscript is its coordinate):
\begin{gather*}
  x_1^j = x_2^j = \dots = x_n^j,\\
\intertext{for $j = 1,\ldots, d-2$,}
  y_1 = x_1^{d-1} = x_2^{d-1} = \dots = x_{n_1}^{d-1} < x_{n_1+1}^{d-1} = \dots = x_n^{d-1} = y_2,
\\
x_1^d < \dots < x_{n_1}^d, \quad\text{and}\quad x_{n_1+1}^d < \dots < x_n^d.
\end{gather*}

Now it remains to calculate the coefficient of the restriction $\partial\colon e' \to e'_{n_1}$. In \cite{vass1988} this coefficient was shown to be $\binom{n}{n_1}$ up to sign for $d=2$; in fact, the computation yields the same result for any $d \geq  2$, because the coordinates $j=1,\ldots, d-2$ remain the same for all points in $e'_{n_1}$ and $e'_n$.  We now show a detailed calculation.

Recall that for any $\sigma\in \Sigma_n$ we orient $\sigma e'$ so that the map $\sigma \colon e'\to \sigma e'$ preserves the orientation. Clearly, this orientation coincides with the orientation given by the form $dx_*^1\wedge \dots \wedge dx_*^{d-1}\wedge dx_1^d\wedge \dots \wedge dx_n^d$ (here $x_*^j$ denotes the common value of $x_i^j$ for $j=1,\ldots, d-2$) if and only if $\sigma$ is an even permutation. Therefore 
the orientation given by the above differential form is actually consistent with $\widehat\ZZ$ coefficients.

We orient $e'_{n_1}$ by the form $dx_*^1\wedge \dots \wedge dx_*^{d-2}\wedge dy_1\wedge dy_2 \wedge dx_1^d\wedge \dots \wedge dx_n^d$.  The boundary~$\partial e'_{n_1}$ corresponds to approaching the equality 
$x_*^{d-1} = y_1 = y_2$ from the side $y_1 < y_2$ and it is therefore oriented in $\partial e'_{n_1}$ by the form $dx_*^1\wedge \dots \wedge dx_*^{d-2}\wedge dx_*^{d-1} \wedge dx_1^d\wedge \dots \wedge dx_n^d$. Thus we obtain:
\[
  \partial e'_{n_1} = \sum_{\sigma\in M_{n_1, n-n_1}} \sgn\sigma\cdot \sigma e',
\]
where the subset $M_{n_1, n-n_1}\subset \Sigma_n$ consists of permutations $\sigma$ such that
\[
  \sigma(1) < \dots < \sigma(n_1)\quad\text{and}\quad \sigma(n_1+1) < \dots < \sigma(n).
\]
Obviously, $|M_{n, n-n_1}|=\binom{n}{n_1}$. For the homology with twisted coefficients we need to calculate:
\begin{equation}
\label{diff-coeff}
\partial \sum_{\tau\in \Sigma_n} \sgn\tau \cdot \tau e'_{n_1} = \sum_{\tau\in \Sigma_n, \sigma\in M_{n, n-n_1}}  \sgn\tau \sgn\sigma \cdot \tau \sigma e' = \binom{n}{n_1} \sum_{\rho =\tau\sigma\in\Sigma_n} \sgn\rho \cdot \rho e'.
\end{equation}
The last equality holds since every $\rho\in\Sigma_n$ can be represented as $\tau \sigma$ for $\tau\in \Sigma_n, \sigma\in M_{n, n-n_1}$ in precisely $\binom{n}{n_1}$ ways.

Since $n$ is a power of a prime $p$, we have the following congruence of polynomials in $t$:
\[
\sum_{n_i=0}^{n} {n \choose n_i} t^{n_i}= (1+t)^n \equiv 1 + t^n \pmod p,
\]
which implies that all the binomial coefficients $\binom{n}{n_1}$ for $0<n_1<n$ are divisible by~$p$.  Hence all the coefficients of the boundary operator, $\sum_{\rho\in \Sigma_n} \sgn\rho\cdot  \rho e'$ are divisible by $p$ and $[e]$ does represent a nonzero homology mod $p$. In fact, we have just proven that $H_{n+d-1}(F'_n(\Reals^d)/\Sigma_n, \pt; \widehat\ZZ) = \mathbb F_p$ with generator $[e]$.

\begin{remark}
With (untwisted) $\mathbb Z$ coefficients the leftmost identity of the formula (\ref{diff-coeff}) still holds but the corresponding cycle $\sum_{\tau\in \Sigma} \tau e'_{n_1}$ may not be divisible by $\binom{n}{n_1}$.
\end{remark}

\subsection{Homologous sections of $\eta^{d-1}$} 

Since the set of generic sections is dense it suffices to prove that a generic section of $\eta^{d-1}$ attains zero somewhere.  Informally, a section without zeros can be made generic without acquiring any zeros. Let $\xi \colon \alpha_n^{\oplus (d-1)}\times F_n(\Reals^d) \to F_n(\Reals^d)$ be a $\Sigma_n$-equivariant vector bundle over the open manifold $F_n(\Reals^d)$. The action of $\Sigma_n$ on $F_n(\Reals^d)$ changes the orientation (by the sign action of $\Sigma_n$) if and only if $d$ is odd; it changes the orientation of $\alpha_n^{\oplus (d-1)}$ if and only if $d$ is even. Hence, it always changes the orientation of $\alpha_n^{\oplus (d-1)}\times F_n(\Reals^d)$.

Consider the linear interpolation between any two generic $\Sigma_n$-equivariant sections $s$ and $t$ of $\xi$, i.e., the section $\lambda s+(1-\lambda)t$. After an appropriate perturbation for $\lambda\in (0, 1)$ this section over $[0,1]\times F_n(\Reals^d)$ becomes transversal to zero and its zero set induces a bordism between the zero sets $Z_s:=s^{-1}(0)$ and $Z_t=t^{-1}(0)$. The zero sets are (not necessarily compact) manifolds. The bordism is not necessarily compact so some care needs to be taken. Before we deal with the compactness issue observe that the bordism is $\Sigma_n$-equivariant, and let us show that the orientations are in accordance with the sign action of $\Sigma_n$.

Let $Z$ be the zero set of a generic section and $\nu(Z)$ its normal bundle in $\alpha_n^{\oplus (d-1)}\times F_n(\Reals^d)$. By transversalitly of $Z$ with the zero section, the bundle $\nu(Z)$ is the restriction of the bundle $\xi\oplus\xi$ to $Z$. The (even-dimensional) bundle $\xi\oplus\xi$ is oriented and its orientation is invariant under the action of $\Sigma_n$, hence the orientation of $Z$ and the action of $\Sigma_n$ on its orientation coincides with that of $\alpha_n^{\oplus (d-1)}\times F_n(\Reals^d)$. 
In a word, the orientation character of $Z$ is actually what we called $\widehat\ZZ$. Hence, we treat zero sets of sections as $\Sigma_n$-equivariant cycles with the sign action of $\Sigma_n$ on their orientation.
We have already exhibited a particular section $s$, for which 
\[
  [Z_s]=[e]\in H_{n+d-1}(F'_n(\Reals^d)/\Sigma_n,\pt;\widehat{\ZZ}).
\] 
For any other $\Sigma_n$-equivariant section $t$ the bordism between $Z_s$ and $Z_t$ is $\Sigma_n$-invariant and 
the action of $\sigma \in \Sigma_n$ on this chain changes its
orientation by $\sgn \sigma$.

Now we deal with the compactness issue.  We first do it via compact support homology. 

The bordism between $Z_s$ and $Z_t$ is not necessarily compact but it is locally compact in $F_n(\Reals^d)$. Hence the zero sections represent well defined classes in the compact support homology and their bordism implies that they are homologous. Since we have already established that $[Z_s] = [e] \not=0$ in
$H_{n+d-1}(F'_n(\Reals^d)/\Sigma_n, \pt; \widehat\ZZ)$, this means that
$[Z_t]$ also represents a nontrivial class and therefore $Z_t$ is not empty as a set.


Now we provide an alternative to the last part of the proof without reference to compact support homology with standard notions of singular homology. We will cut out a submanifold with boundary $(M, \partial M)$ from $F_n(\Reals^d)$ so that the bordism between $Z_s$ and $Z_t$ restricts to a compact bordism between the relative cycles $[Z_s\cap M]$ and $[Z_t\cap M]$, and the pullback of the class $[e]$, represented by $[Z_s\cap M]$, does not vanish in $H_{n+d-1}(M/\Sigma_n,\partial M/\Sigma_n; \widehat{\ZZ})$.

Observe that the pair $(F'_n(\Reals^d), \pt)$ was actually realized as follows: we considered $F_n(\Reals^d)$ as a subset of $\mathbb R^{nd}$, such that its complement $X=\mathbb R^{nd}\setminus F_n(\Reals^d)$ is a union of several linear subspaces of $\mathbb R^{nd}$.
If we compactify $\mathbb R^{nd}$ to the sphere $S^{nd}$ with a point at infinity, then $X$ gets compactified to $X'\subset S^{nd}$. By excision, we see that the homology of the pair $(F_n'(\Reals^d)/\Sigma_n, pt)$ is isomorphic to the homology of the pair $(S^{nd}/\Sigma_n, X'/\Sigma_n)$. 
Now we observe that $X'$ is a union of a finite number of equatorial subspheres of codimension $d$ in $S^{nd}$. Hence $X'$ is a closed set and it has a neighborhood $N\supset X'$ that can be deformed equivariantly onto $X'$.

Now we choose a $nd$-dimensional, $\Sigma_n$-invariant smooth manifold with boundary $M\subset F_n(\Reals^d)$ so that $\partial M$ is contained in~$N$. 
Since the deformation retraction is equivariant, the homology map $H_*(X'/\Sigma_n; \widehat \ZZ) \to H_*(N/\Sigma_n; \widehat \ZZ)$ is an isomorphism. Put $\bar M = S^{nd}\setminus \inte M\subseteq N$ and consider the natural homology maps
\[
\begin{CD}
H_*(S^{nd}/\Sigma_n, X'/\Sigma_n; \widehat \ZZ) @>{f}>> H_*(S^{nd}/\Sigma_n, \bar M/\Sigma_n; \widehat \ZZ) @>{g}>> H_*(S^{nd}/\Sigma_n, N/\Sigma_n; \widehat \ZZ).
\end{CD}
\]
Their composition $g\circ f$ is an isomorphism and therefore $f$ is an injection. Again by excision, we also have 
$$
H_*(S^{nd}/\Sigma_n, \bar M/\Sigma_n; \widehat \ZZ) = H_*(M/\Sigma_n, \partial M/\Sigma_n; \widehat \ZZ).
$$ 
Therefore the manifold $Z_s \cap M$ with boundary $Z_s\cap \partial M$ represents a nonzero class in $H_{n+d-1}(M/\Sigma_n, \partial M/\Sigma_n; \widehat\ZZ)$. The equivariant bordism between $Z_s$ and $Z_t$ restricts to a compact equivariant bordism between relative homology cycles $[Z_s\cap M]$ and $[Z_t\cap M]$ in $(M/\Sigma_n, \partial M/\Sigma_n)$, with coefficients $\widehat \ZZ$, hence $[Z_t\cap M]$ also represents a nontrival relative homology in $(M/\Sigma_n, \partial M/\Sigma_n)$ and the set $Z_t\cap M$ cannot be empty. This establishes Theorem~\ref{top}. 

\begin{remark}
  The last part of the argument with the pair $(M,\partial M)$ is
  valid for any coefficient system as it only relies on excision. This
  argument can be used to define the Euler class of $\xi$ with
  coefficients in $\widehat{\ZZ}$ without reference to compact support
  homology by taking Poincar\'e duality with the pair $(M,\partial
  M)$.
\end{remark} 

\section*{Remarks and questions}

Recall the waist of the sphere inequality of \cite{gr2003, mem2009}:

\begin{theorem} 
For any continuous map $f\colon \mathbb S^d \to \Reals^k$, there exists $z \in \Reals^k$ such that,  
\[
  \vol(f^{-1}(z)+t) \geq \vol(\mathbb S^{d-k}+t),
\] 
for all $t$, where $\mathbb S^{d-k}$ is an equatorial sphere inside $\mathbb S^d$.
\end{theorem}

Notice that $t$ is not assumed to be small in the statement of the theorem.

The proof of this theorem relies on ideas from the celebrated \emph{localization} technique. The main new tool to perform a localization-type argument (called a \emph{pancake decomposition} in \cite{mem2009}) when the image is not one-dimensional is Theorem~\ref{gro}.

It will be interesting to use the methods of Section~\ref{section:transport} and the analytical content of Gromov's proof to obtain a \emph{comparison} waist inequality for manifolds of positive curvature with the normalized Riemannian volume. What if we drop the curvature constraint? Are there other variational problems on the space of cycles or related geometric inequalities that can be approached through variational or topological problems on configuration space or related spaces through optimal transport?  

Can we say something about the combinatorics of power diagrams for some natural continuous functionals? Can we say anything about the measure $\nu=\frac{1}{n}\sum \delta_{x_i}$  as $n$ approaches infinity, or about its weak limit? In the case of a cube, for any $n$ there are many ``spicy chicken'' partitions, is this the generic case or an exceptional one?


Meanwhile, the Nandakumar--Ramana Rao conjecture remains open for six pieces in the plane.

\subsection*{Acknowledgments}
We thank Arseniy~Akopyan, Imre~B\'ar\'any, Pavle~Blagojevi\'c, Sylvain~Cappell, Fred~Cohen, Daniel~Klain, Erwin~Lutwak, Yashar~Memarian, Ed~Miller, Gabriel~Nivasch, Steven~Simon, and Alexey~Volovikov for
discussions, useful remarks, and references.  We also thank an
anonymous referee for encouraging us to merge our papers and for
his/her enthusiasm towards the \emph{chicken nuggets} description of
Corollary~\ref{spicy-chicken}.

\section{Appendix of 2016}

\subsection{Sparrows at cannons.}

Recently we became aware of an unrefereed popular exposition \cite{z2015} in which G\"unter Ziegler discusses the results first discovered in this paper and inaccurately criticizes one of our proofs. Referring to the topological content of the published version \cite{ahk2013} of the current paper, in \cite{z2015}, Ziegler asserts: ``\textsl{There seems to be no way to make their approach, which would need a relationship between the Euler class and homology classes induced by generic cross-sections, rigorous and complete.}'' This is incorrect even at a superficial level. The proof written above avoids invoking the Euler class, as we remarked in the introduction and body of this paper. In this appendix we address these criticisms and make some clarifications. We remark that the text before the Appendix was not changed from the published version of this article. 

\medskip

Before going into the details of  this appendix we make a tangential remark. The new contribution of the paper  \textsl{Convex equipartitions and equivariant obstruction theory} \cite{bz2012} coauthored by G\"unter Ziegler and Pavle Blagojevi\'c is topological. They show that the conditions of Theorem~\ref{top} are necessary. In the case of the plane, Arone \cite{arone} had previously shown that, if $n\neq p^k,2p^k$, for a prime $p$, then there exist $\Sigma_n$-equivariant maps $F_n(\mathbb R^2) \to \mathbb S^{n-2}$.  The result in \cite{bz2012} takes care of the cases $n=2 p^k$, $d=2$ and the cases $n\neq p^k$ for general dimension. The \textsl{convex equipartition} part of their paper is taken verbatim from this paper in its published version \cite{ahk2013}. The Nandakumar--Ramana Rao conjecture remains open for any number that is not a prime power and the results of \cite{bz2012} imply that the technique of the current paper does not extend to the non-prime power case.


\subsection{Organization of this appendix}

In Section \ref{FV} of this appendix we review the critical comments from \cite{z2015} and from (the last version\footnote{These claims were not part of the first arXiv version of \cite{bz2012}.} of) \cite{bz2012}. We explain why they do not apply to this paper. Notwithstanding the fact that the authors of~\cite{bz2012} did not provide any argument or reference to back their claims, their comments have prompted us to scrutinize our writing and consult algebraic topologists. Our colleagues agree with our arguments and have suggested that we
clarify the term ``compactly supported homology'' which might be misleading. In order to make more explicit that we do not rely on
homotopy invariance of this homology theory, we also explain in more detail the two different compactifications and two forms of excision
used in the proof of Theorem~\ref{top}.

In Section~\ref{argument} we clarify issues about the type of homology and make some well known remarks on transversality. To provide more details we outline the proof of Theorem~\ref{top} in Section~\ref{outline}. The part of our proof that seems closer to the criticism relies on homological considerations that, in turn, rely on some standard geometric topology. In Section~\ref{homology} we provide references for the homological part and a detailed construction of the manifold with boundary, denoted by $M$, that we use as a compact version of the configuration space.
We give two constructions  using the PL topology and using the smooth topology.

We use $\mathcal{S}$ to refer to a piecewise-linear sphere, $S^N$ to a topological sphere, and $\Sphere^N$ to a sphere as a metric space.  We follow the same typographical convention for other objects.



\subsection{The Fuks--Vasiliev proof}\label{FV}

The proof of Theorem~\ref{top} that relies on Euler class considerations is not due to us. The case $n=2^k$, $d=2$ of Theorem~\ref{top} follows the study of the braid group in the paper \cite{fuks1970} of D.~B.~Fuks. The cases $n=p^k$, $d=2$ of Theorem~\ref{top} appeared on page~186 of V.~Vasiliev's paper \cite{vass1988}, where the theorem is written in terms of the Euler class and is applied to the study of topological complexity (in the sense of Smale) of an algorithm that computes all the roots of a polynomial. In more standard terms, it was used to estimate the Lusternik--Schnirelmann category of $F_n(\Reals^2)/\Sigma_n$ from below. The observation that these proofs generalize to any dimension for $n=p^k$ can be found in print as Lemma~5 in \cite{kar2009}, where the case $n=p$ is obtained by a different method (Lemma~4 of the same paper). In the preprints \cite{huar2010} and \cite{kar2010} we wrote Vasiliev's proof in more detail, explicitly stating that no new content was being added. The criticism of \cite{bz2012} was aimed at this proof, as can be seen from the acceptance and publication dates. 

The skepticism expressed in \cite{z2015} towards the proofs of Fuks and Vasiliev (which are attributed to us) continues in a footnote:``\textsl{Karasev et al. also set out to show the non-existence of the equivariant map $F(\Reals^d, n) \to S^{n-2}$ in the case when $n$ is a prime power. For this, they intend to show that the Euler class with twisted coefficients of a natural vector bundle over the open manifold $F(\Reals^d, n)/\Sigma_n$ is non-zero. However, the relationship between a generic cross-section and the Euler class of the bundle via Poincar\'e duality breaks down over open manifolds. For a more detailed discussion, see [7, p. 51].}'' It refers to page~51 of the paper \cite{bz2012} for a ``\textsl{more detailed}'' discussion. In the paper \cite{bz2012} the same claim is repeated, and the reader is referred to page~48 of the book \cite{p2007} to justify the fact that (co)homology with compact support is not a homotopy invariant, a fact that can be explained in one sentence: Two Euclidean spaces of different dimension are homotopically equivalent, yet $H_i^{\mathrm{cl}}(\Reals^d)=\tilde H_i(S^d)=\ZZ$ if $i=d$ and $0$ otherwise, so $H_i^{\mathrm{cl}}$ is not a homotopy invariant. This fact does not contradict anything stated in the published version of this paper \cite{ahk2013}.  Blagojevi\'c and Ziegler conclude from this elementary fact without further argument or reference to other works that (see p.~51 \cite{bz2012}) ``\textsl{the fundamental classes of singular sets of two different generic cross sections do not need to coincide in $H^{\mathrm{cl}}_*(M)$.}''  


The fact that compactly supported (co)homology is not a homotopy invariant does not provide evidence supporting the claim that the proofs of Fuks and Vasiliev, nor the one given here, are incorrect. In fact, what we showed was that there is a manifold with boundary $M$ contained in $F(\Reals^d, n)$ for which the fundamental classes of singular sets of two different cross sections coincide in $H_*^{\Sigma_n}(M,\partial M; \cdot)$. 



In the last remark of the published version of this paper \cite{ahk2013} we suggested that the strategy used here could be seen as a justification of the Fuks--Vasiliev argument for the nontrivially of the Euler class. This is a slightly stronger statement than Theorem~\ref{top} and some authors use this approach as the definition of the Euler class. In retrospect we prefer to see this remark as non-rigorous: Given a vector bundle and its orientation sheaf of coefficients, the corresponding Euler class is defined by pulling back the corresponding Euler class from a principal bundle of the non-oriented Grassmannian with coefficients in its orientation sheaf.

We are not aware of any previous criticism of the validity of Vasiliev's result. We could not find a satisfying reference to identify the Euler class as a Poincar\'e dual in the case of $n$ power of an odd prime; yet when we asked algebraic topologists, they seem confident that the standard argument extends to this generality. It is perhaps worth mentioning that the technique used in this paper is reminiscent of related ideas of Vasiliev (see for instance his plenary 1994 ICM talk and the homonymous book \cite{vass1994}, in particular, Chapter 1 and Appendices 4 and 5). 

Observe that in the case $n=2^k$, the second argument explained in Section 5.4 of this paper does yield the claimed non-triviality of the Euler class with a textbook argument. Indeed, the classifying map of the bundle restricted to the wreath product of spheres mentioned in Section 5.4 factors through the inclusion map into configuration space, i.e.  $S^{d-1}\wr\dots\wr S^{d-1}/\Sigma_{n}^{(2)} \to F(\Reals^d, n)/\Sigma_{n}^{(2)} \to G_{(d-1)n}$, where $G_{(d-1)n}$ is the non-oriented infinite Grassmannian. Since $S^{d-1}\wr\dots\wr S^{d-1}/\Sigma_{n}^{(2)}$ is a compact smooth manifold, the argument sketched there yields that the Euler class modulo two\footnote{In this case the Euler class coincides with the top Stiefel--Whitney class.} 
$e(\xi_{|{S^{d-1}\wr\dots\wr S^{d-1}/\Sigma_{n}^{(2)}}})$ is not zero via Poincar\'e duality on $S^{d-1}\wr\dots\wr S^{d-1}/\Sigma_{n}^{(2)}$. Since $i^*(e(\xi))=e(\xi_{|_{S^{d-1}\wr\dots\wr S^{d-1}/\Sigma_{n}^{(2)}}})$, where $e$ denotes the Euler class and $i$ is the inclusion map, we might conclude that $e(\xi)\neq 0$ in the $\Sigma_{n}^{(2)}$-equivariant cohomology, and therefore in $\Sigma_{n}$-equivariant cohomology. 

A similar argument might yield the nontriviality of the Euler class substituting the product of spheres with the manifold with boundary that appears in our proof of Theorem ~\ref{top}. Taking the cap product of the test cycle $Z_s$ with $(M,\partial M)$ yields a cohomology class in $H^{(d-1)(n-1)}_{\Sigma_n}(M; \cdot)$. Our proof implies that this class coincides with the cap product of any other generic zero set $Z_t$ with $(M,\partial M)$. Moreover, in our proof we can take $M$ arbitrarily large in the sense, for any compact set $K \subset F(\Reals^d, n)$, we can construct $M$ containing~$K$.   For the spicy chicken theorem and similar potential applications in discrete and convex geometry, Theorem~\ref{top} suffices. We leave the Euler class discussion to algebraic topologists. 

\subsection{Clarifications on homology}\label{argument}


Our colleagues point out that the term ``compact support homology'' in our proof might be misleading. Ordinary singular chains are finite and therefore have compact support by definition. We used this term in analogy with the textbook notion of ``compact support cohomology.'' It is better to call the homology we use ``locally finite homology.''  This is defined through possibly infinite singular chains such that every compact set intersects only a finite number of singular simplicies from the chain. 
In our case, this type of homology will be equal to the relative (equivariant) homology of a compactification, provided the complement of our open manifold is a deformation retract of its neighborhood. This was mentioned in the published version of this paper~\cite{ahk2013}. For non-pathological spaces this homology coincides with Borel--Moore homology as defined through sheaves.

The second clarification we should make is that in our text we take the quotient by the symmetric group in some parts of the argument. If the action is free, the homology of the quotient is essentially the same as the equivariant homology of the space. In our case the action is free but the complement in the compactification has nontrivial stabilizers, so we prefer to work with equivariant homology and refer the reader to the book \cite{bred1967} for the basic facts that we use. 

Notice that the issues of transversality are not delicate. Namely, if a finite group $G$ acts freely on a manifold $X$, any equivariant section $f\colon X\to_G V$ (here $V$ is a representation of $G$) can be made transversal preserving equivariance. Indeed, the existence of a transversal equivariant section $\hat{f}$ that is arbitrarily close to $f$ reduces to considering the quotient bundle $(X\times V)/G$ and its sections over the manifold $X/G$, where the total space of the bundle is $(X\times V)/G$. So the question reduces to standard transversality theory, without the need of delicate issues of equivariant transversality in the presence of non-free group actions. 

Given a continuous equivariant section of the vector bundle without zeros, we approximate it in strong topology by a smooth equivariant section, so that the latter section has no zeros, and thus is transversal to zero. Therefore we can speak about the smooth and transversal-to-zero equivariant sections of the bundle in question. We stress once again that, for a smooth and transversal-to-zero section $t$, the zero set $Z_t = t^{-1}(0)$ is a closed submanifold (without boundary) of the configuration space that is not necessarily compact and may be open.

\subsection{Outline of the proof of Theorem~\ref{top}}\label{outline}

Let us outline the proof of the main topological lemma of the paper. In particular, we do it to emphasize that none of the following steps assume nor use homotopy invariance of compactly (finitely) supported homology. 

\begin{enumerate}


\item The aim of the proof.
 
We want to show that the zero set $Z_t=t^{-1}(0)$ of any equivariant section $t \colon F_n(\Reals^d) \to \xi$ is not empty. The transversality argument allows us (as explained above) to consider smooth and generic section in place of $t$. 

\item Test section.

The test section analyzed in Section~5.6 of this paper is transversal to the zero section. The computation there shows that $Z_s=s^{-1}(0)$, represents a non-trivial element in the singular equivariant homology of the pair $(F_n'(\Reals^d),\pt)$.

\item Homology isomorphism. 

In section 5.7 we claimed that there exists a $\Sigma_n$-invariant smooth manifold with boundary $M$ such that the pairs $(F_n'(\Reals^d),\pt)$ and $(M,\partial M)$ have isomorphic equivariant homology groups and $\partial M$ is transversal to  $Z_s$. This isomorphism is established by the excision and homotopy invariance of the equivariant homology. Below we provide two constructions yielding the manifold with boundary $M$, one construction uses smooth topology and another one uses piecewise-linear topology. Since our objective is a homology isomorphism, either approach suffices to show that $Z_s \cap (M,\partial M)$ also represents a non-trivial relative homology class in the singular equivariant homology of the pair $(M, \partial M)$.

\item Homologous cycles.

By perturbing the linear interpolation between the sections $t$ and $s$, to make it transversal to the zero section, we obtain a compact bordism with boundary between the manifolds with boundary $Z_s \cap (M,\partial M)$ and $Z_t \cap (M,\partial M)$. This results in establishing that $Z_t \cap (M,\partial M)$ represents a nonzero element in the equivariant homology of the pair $(M,\partial M)$, implying that $Z_t \cap (M,\partial M)$ (and hence $Z_t$) is not empty. Details about the transversality of the considered intersections are provided below.

\end{enumerate}

\subsection{Homology isomorphism and reduction to manifold with boundary.}\label{homology}
 
The last section of our text is somewhat confusing because we deal with two different types of excisions and two compactifications. Below we make things more precise.

The second time we refer to excision is the classical one from singular homology, namely $H_n(A,B)=H_n(A\setminus C,B\setminus C)$ provided that the interiors of $B$ and $A\setminus C$ cover $A$. The form of excision that we are using in the first part of the proof can be found in the book \cite{may1999}, on page 108 (Section 14.2): If $A\to Y$ is a cofibration, then all homology theories of $(Y, A)$ and $(Y/A, \pt)$ coincide.  In our proof, we claimed  that the $\Sigma_n$-equivariant homology groups of  $(F_n'(\Reals^d), \pt;\widehat \ZZ)$ are the same as those of  $(S^{nd}, X';\widehat \ZZ)$. This is intuitively obvious and easy to verify: Provided that $X' \to S^{nd}$ is a $\Sigma_n$-equivariant cofibration, setting $Y=S^{nd}$ and $A=X'$ yields the claimed isomorphism. 


It is well known that  $A \to Y$ is a cofibration if there is a retraction from $Y \times I \to (A\times I)\cap(Y \times {0})$, see Chapter~6 of \cite{may1999}. 
In the piecewise-linear argument, the cofibration property follows from the following general fact: If $A$ is a piecewise-linear polyhedron embedded in a piecewise-linear manifold $Y$ then the pair $(Y,A)$ is a cofibration. We need an equivariant version of this, so we provide some details below.

\subsubsection{The setup}
Before going into the proofs we set up further terminology. We can realize the one-point compactification of $\Reals^{nd}$ geometrically by the inverse stereographic projection $\pi^{-1}\colon \mathbb{R}^{nd} \cup \infty \to S^{nd}$. Here the origin of $\Reals^{nd}$ is mapped to the south pole of the sphere and the point at infinity to the north pole. Notice that the stereographic projection is a $\ZZ_2$-equivariant map, with the $\ZZ_2$ action on $S^{nd}$ given by the reflection on the horizontal equator $S^{nd-1}\subset S^{nd}$ and the $\ZZ_2$ action on $\Reals^{nd}-\{0\}$ given by the inversion with respect to the unit sphere $S^{nd-1}\subset \Reals^{nd}$, i.e., $\forall x \in \Reals^{nd}-\{0\}\colon x \mapsto \frac{x}{|x|^2}$. 

Since every linear subspace in $\Reals^{nd}$ is invariant under the inversion on the unit sphere, the fat diagonal $X\subset\Reals^{dn}$ is also invariant, hence $X'\subset S^{nd}$ is invariant under the reflection on the equatorial sphere. In the following, we denote both of these involutions by $i$. Abusing notation further, we denote by $B^{nd}$ both the $nd$-dimensional closed ball of unit radius and the lower hemisphere of $S^{nd}$. Similarly $\partial B^{nd}=S^{nd-1}$ will denote both the unit sphere in $\Reals^{nd}$ and the equator of $S^{nd}$.  The action of $\Sigma_n$ on $\Reals^{nd}$ and on $S^{nd}$ can be thought of as the restriction of the action of $\Sigma_n$ on $\Reals^{nd}\times \Reals$, where $\Sigma_n$ acts trivially on the last factor.

\subsubsection{Piecewise-linear topology}

The classic text of Rourke and Sanderson \cite{rs1972} is a great introduction to piecewise-linear topology. We will refer to this book several times in this section. We quickly recall the necessary definition of a polyhedron, a detailed discussion can be found in the first chapter of \cite{rs1972}. Given a point $x$ and a set $L$ in $\Reals^m$ denote by $xL$ the cone with vertex $x$ and base $L$, this is the set of convex combinations $xL:=\{\lambda x+(1-\lambda) y\in xL: \lambda\in [0,1], y \in L\}$. A subset $K \subset \Reals^m$ is said to be a polyhedron if, for every $x\in K$, there exists a neighborhood $U \subset K$ of $x$ and a set $L$ such that $xL=U$.  The second chapter of \cite{rs1972} explains how polyhedra can be manipulated via simplicial complexes, particularly the concept of gluing that appears below is discussed in Exercise~2 of 2.27. The third chapter discusses regular neighborhoods of polyhedra. The concept of collapse (or Whitehead simple homotopy) appears in that chapter in relation to regular neighborhoods and will be of relevance to us.

We now construct a PL-sphere $\mathcal{S}$ and a polyhedron $\mathcal{X}'$ contained in $\mathcal{S}$ together with a homeomorphism from $S^{nd} \to \mathcal{S}$ that sends $X' \to \mathcal{X'}$ and is equivariant both with respect to $\ZZ_2$ acting by reflection on the horizontal equator and to the action of $\Sigma_n$. We define $\mathcal{S}$ by gluing two disjoint copies $K_1$ and $K_2$ of $[-1,1]^{nd}$ on their boundary following the way the two copies of $B^{nd}$ are glued on their boundary to make $S^{nd}$.  Formally,  consider disjoint copies of the standard homeomorphism $\phi_1 \colon B^{nd} \to  K_1$ and $\phi_2 \colon B^{nd} \to  K_2$ and let $\mathcal{S}$ be the attaching space of $(\phi_2 \circ i \circ \phi_1^{-1})\colon \partial K_1 \to \partial K_2$ which is a $PL$ homeomorphism. The map $\phi_1$ extends from the ball $B^{nd}$ to a $\ZZ_2$-equivariant map of the whole sphere. Denote this map also by $\phi \colon S^{nd} \to \mathcal{S}$, finally let $\mathcal{X'}:=\phi(X')$.

To check that $\phi(X')$ is a polyhedron, first notice that by construction it is enough to show it in the restriction to $K_1$, but here $\phi(X')$ is exactly the fat diagonal restricted to the cube. Consider the linear space $H_{12}\subset \Reals^{nd}$ of codimension~$d$ given by the equalities $x_1=x_2$.  Directly from the definition we see that $K_1\cap H_{12}$ is a subpolyhedron of $K_1$, that is there is a triangulation of $K_1$ such that $H_{12}\cap K_1$ is its sub-triangulation. Consider the orbit of this triangulation under the action of $\Sigma_n$, observe that the orbit of $H_{12}$ is the fat diagonal in $K_1$. So we have a family of triangulations $\{\mathcal T_g\}_{g\in \Sigma_n}$ of $K_1$, and the subsets 
$$
\bigcap_{g\in \Sigma_n} \sigma_g,\quad \sigma_g\ \text{is a face of}\ \mathcal T_g 
$$
compose a cellular decomposition of $K_1$ so that the fat diagonal is its sub-cellular decomposition. Then this decomposition can be barycentrically subdivided to give a $\Sigma_n$-equivariant triangulation of $(K_1,\partial K_1)$ such that $(K_1\cap X')$ is a $\Sigma_n$-equivariant triangulated subpolyhedron contained in it.

Finally we extend this triangulation to all of $\mathcal{S}$ using the involution that maps $K_1$ to $K_2$ leaving their boundaries invariant. This finishes the construction. Now, Proposition 3.10 in \cite{rs1972} states that in this situation a regular neighborhood $\mathcal{N}$ of $\mathcal{X'}$ in $\mathcal{S}$ is a PL-manifold with boundary, the construction there allows us to choose $\mathcal{N}$  in a $\Sigma_n$-invariant way. Then Corollary~3.30 in \cite{rs1972} asserts that if $\mathcal{N}$ is the regular neighborhood of $\mathcal{X'}$ and it is a PL-manifold with boundary, then it collapses onto $\mathcal{X'}$. A simple analysis of the construction in the proof of Corollary ~3.30 \cite{rs1972} shows that this construction is also $\Sigma_n$-invariant. It follows that $X' \to S^{nd}$ is an equivariant cofibration, since the regular neighborhood collapses onto the subspace equivariantly. We conclude that any homology $H^{\Sigma_n}_*(F_n'(\Reals^d),\pt; \cdot)$ is isomorphic to its respective $H^{\Sigma_n}_*(S^{nd}, X'; \cdot)$, moreover, the equivariant relative cycle $Z_s$ represents a non-zero homology class in $H^{\Sigma_n}_*(S^{nd}, X'; \cdot)$.  The collapse from $\mathcal{N}$ to $\mathcal{X'}$ yields that the latter groups are isomorphic to $H^{\Sigma_n}_*(\mathcal{S},\mathcal{N}; \cdot)$, and if we choose $N$ small enough, the relative cycle $Z_s \cap (\mathcal{S},\mathcal{N})$ represents the image of the homology class $[Z_s]$.\\

The next step in our proof is to choose a manifold with boundary on $\mathcal{N}$. Below we provide details on how to do so using smooth topology. With the language that we have developed it is also easy to construct a piecewise linear manifold with boundary: 

$\bullet$ Using Proposition 3.10 of \cite{rs1972} again, we can take $\mathcal{M}=S^{nd}-\inte(\mathcal{N}')$ where $\mathcal{N}' \subset \mathcal{N}$ is another regular neighborhood of $\mathcal{X}'$.

In both, smooth or piecewise models, we have isomorphims, 
$$
H^{\Sigma_n}_*(F_n(\Reals^d),\pt ; \widehat \ZZ)\to H^{\Sigma_n}_*(\mathcal{S}, \mathcal X'; \widehat \ZZ)\to H^{\Sigma_n}_*(\mathcal{S}, \mathcal{N}; \widehat \ZZ)\to H^{\Sigma_n}_*(\mathcal{M},\partial \mathcal{M}; \widehat \ZZ).
$$

By similar reasoning, if $\mathcal{N}$ is small enough then the restriction map $H^{\Sigma_n}_*(Z_s,\pt ; \widehat \ZZ)\to H^{\Sigma_n}_*(Z_s,  Z_s \cap \mathcal{N} ; \widehat \ZZ)$ is an isomorphism, so the class $[Z_s \cap (\mathcal{M},\partial \mathcal{M})] \in H^{\Sigma_n}_*(\mathcal{M},\partial \mathcal{M}; \widehat \ZZ)$ is the image of the non-trivial class $[Z_s] \in H^{\Sigma_n}_*(F_n(\Reals^d),\pt ; \widehat \ZZ)$. The free $\Sigma_n$ action on the pair $(\mathcal M, \partial \mathcal M)$  induces a free $\Sigma_n$ action on $(\mathcal M,\partial \mathcal M) \times I$ (acting trivially on the second component). Denote by $\hat{\xi}$ the associated bundle and consider the $\Sigma_n$-equivariant section $\Gamma \colon (\mathcal M,\partial \mathcal M)  \times I \to \hat{\xi}$, $\Gamma(x,\lambda):=\lambda s(x)+(1-\lambda) t(x)$. Since the action of $\Sigma_n$ is free on $(\mathcal M,\partial \mathcal M) \times I$ we can pass to the quotient by $\Sigma_n$ and then perturb $\Gamma/\Sigma_n$, to a section $\hat{\Gamma}/\Sigma_n$ that is transversal to the zero section of the bundle. We can also assume that for all $x$, $\hat{\Gamma}(x,1)=s(x)$ and $|\hat{\Gamma}(x,0)-t(x)|$ is arbitrary small. Moreover, if we take a fine triangulation of $\mathcal{M}$, we can perturb the facets of $\partial \mathcal{M}$ smoothly and equivariantly so that $\hat{\Gamma}^{-1}(0)$ is transversal to the facets of $\partial \mathcal M$. This perturbation can be done so that the smooth facets are arbitrarly close to the flat ones. Then $\hat{\Gamma}^{-1}(0)$ is an equivariant compact singular bordism between the equivariant relative cycle  $Z_s \cap (\mathcal M,\partial \mathcal M)$ and an equivariant relative cycle that is arbitrary close to $Z_t \cap (\mathcal M,\partial \mathcal M)$. A triangulation of this bordism into a singular chain then witnesses that the two boundaries of the bordism are equivariantly homologous implying that $Z_t \cap (\mathcal M,\partial \mathcal M)$ is not empty.

\subsubsection{Smooth topology}
We now construct $M$ using elementary smooth topology instead of piecewise linear tools. Notice that the fact that the pair $(S^{nd},X')$ is a cofibration is topological, so it follows from the previous subsection by taking a homemorphism from the PL model to the smooth one. Now let 
$$
H^{ij}:=\{(x_1,x_2\ldots x_n,t): x_i \in \Reals^{d},  t\in \Reals,  x_i= x_j\}\subset \Reals^{nd+1},
$$ 
then $X=\Reals^{nd} \cap (\cup_{i\neq j} H^{ij})$ is the fat diagonal and $X'=S^{nd} \cap  (\cup_{i\neq j} H^{ij})$. Our configuration space is equivariantly diffeomorphic to $S^{nd}\setminus X'$ with the stereographic projection from $S^{nd}$ to $\mathbb R^{nd}$.

Consider the function $f \colon \Reals^{nd+1} \to \Reals$, 
$$
f(x)= \prod_{i<j} \dist(x,H^{ij})^2,
$$
This function is non-negative, smooth, and invariant with respect to the $\Sigma_n$ action. This function restricted to the sphere $S^{nd}$ vanishes precisely at $X'$.  Let $\nabla f$ be the gradient of the restriction of $f$ to $S^{nd}$, if we can show that the gradient of $f$ does not vanish on $f^{-1}(0,\tau)$  for sufficiently small $\tau>0$, then the flow along $-\nabla f$ (restricted to $S^{nd}$) is a $\Sigma_n$-equivariant deformation retraction of $N:=f^{-1}(0,\tau)$ onto $X'$. 

So once the existence of $\tau$ is established, we can define the manifold with boundary by $M:=f^{-1}[\tau,\infty) $ and follow the homological argument after the bullet point in the previous subsection. The only property that we use from $f$ (other than the fact that it vanishes on $X'$ and only on $X'$) is that it is given by a polynomial. Indeed, since we work on the sphere $S^{nd}$ that is given by algebraic equations, a simple vector calculation implies that the vanishing of $\nabla f$ is an algebraic condition. So the set of critical points of $f$ is algebraic, and $f$ takes the set of critical points to the set of critical values, which therefore is a semialgebraic\footnote{A set is \emph{semialgebraic} if it is given by algebraic equations and inequalities.} subset of the positive reals (see~\cite{rag1998} for the background). Sard's theorem, in turn, states that the set of critical values has zero measure. Since semialgebraic subsets of the reals of zero measure must be finite, we conclude that there exists a value $\tau>0$ so that for all $t<\tau$, $\nabla f$ is not zero.

\subsection{In textbook terms: using relative homology from the start}

Actually the most delicate part of the argument, the one-point compactification and excision, is not completely necessary. Here we describe one more simplification that allows us to argue in textbook terms without locally finite homology, just in terms of equivariant singular homology. Using the PL approach, we could consider not the configuration space of $n$-tuples in $\Reals^d$, but the configuration space of $n$-tuples in the cube $[-1,1]^d$, where we will work relative to the boundary of the cube.

The model for the configuration space in this case is the closed cube $Q := [-1,1]^{nd}$ relative to its subspace 
$$
Y = (X\cap Q)\cup \bd Q,
$$
where $X$ is the fat diagonal of $(\Reals^d)^n$. The permutation group $\Sigma_n$ (and therefore the alternating group $A_n$) acts freely on the complement $Q\setminus Y$. There is an equivariant CW-decomposition of $Q$ relative to $Y$ similar to the Fuks decomposition (and combinatorially equal to it) except that now all coordinates are in~$[-1,1]$, and the $n!$ cells of dimension $n+d-1$ are pasted to the fat diagonal and the boundary of $Q$. This cell decomposition has $(n+d-2)$-dimensional cells given by the conditions,
\begin{gather*}
  x_1^j = x_2^j = \dots = x_n^j,\\
\intertext{for $j = 1,\ldots, d-1$, and}
-1\leq x_1^d < \dots < x_{n_1}^d\leq x_{n_1+1}^d < \dots < x_n^d\leq 1.
\end{gather*}

We compute the relative equivariant homology $H^{\Sigma_n}_*(Q, Y; \widehat \ZZ)$ and identify the zero set of the test section in this homology in precisely the same manner as we did for the locally finite homology of the configuration space. The procedure and the result will be naturally the same as those in Section 5.6. Then we can find a $\Sigma_n$-invariant regular neighborhood $N$ of $Y$, pass to the PL-manifold $M = Q\setminus \inte Y$ with free group action, and then work in the textbook terms with the free $\Sigma_n$-action on the pair $(M, \partial M)$. Again, we could even use $A_n\subset \Sigma_n$ for odd primes to avoid non-constant coefficients.

\end{document}